\documentclass[11pt]{article}
\usepackage{amsmath}
\usepackage{amssymb}
\usepackage{amsthm}
\usepackage{mathabx}
\usepackage[usenames]{color}
\usepackage{amscd}
\usepackage{dsfont}
\usepackage{indentfirst}
\usepackage{tabu}

\usepackage[colorlinks=true,linkcolor=blue,filecolor=red,
citecolor=webgreen]{hyperref}
\definecolor{webgreen}{rgb}{0,.5,0}

\def\C{{\mathds{C}}}

\def\N{{\mathds{N}}}
\def\Z{{\mathds{Z}}}

\def\1{{\bf 1}}

\def\lcm{\operatorname{lcm}}

\newtheorem{theorem}{Theorem}[section]

\newtheorem{lemma}[theorem]{Lemma}
\newtheorem{cor}[theorem]{Corollary}

\begin{document}

\title{{\bf Counting subrings of the ring $\Z_m \times \Z_n$}}
\author{L\'aszl\'o T\'oth \\ \\ Department of Mathematics, University of P\'ecs \\
Ifj\'us\'ag \'utja 6, 7624 P\'ecs, Hungary \\ E-mail: {\tt ltoth@gamma.ttk.pte.hu}}
\date{}
\maketitle

\centerline{J. Korean Math. Soc. {\bf 56} (2019), No. 6, pp. 1599--1611}

\begin{abstract} Let $m,n\in \N$. We represent the additive subgroups of the ring $\Z_m \times \Z_n$, which are also
(unital) subrings, and deduce explicit formulas for $N^{(s)}(m,n)$ and $N^{(us)}(m,n)$, denoting the number of subrings of the ring
$\Z_m \times \Z_n$ and its unital subrings, respectively. We show that the functions $(m,n)\mapsto N^{(s)}(m,n)$ and
$(m,n)\mapsto N^{(us)}(m,n)$ are  multiplicative, viewed as functions of two variables, and their Dirichlet series can
be expressed in terms of the Riemann zeta function. We also establish an asymptotic formula for the sum $\sum_{m,n\le x} N^{(s)}(m,n)$,
the error term of which is closely related to the Dirichlet divisor problem.
\end{abstract}



\section{Motivation and preliminaries} \label{Section_Motivation}

Throughout the paper we use the following notation: $\N:=\{1,2,\ldots\}$, $\N_0:=\{0,1,2,\ldots\}$;
the prime power factorization of $n\in \N$ is $n=\prod_p p^{\nu_p(n)}$, the product being over the primes $p$,
where all but a finite number of the exponents $\nu_p(n)$ are zero; $\gcd(m,n)$ and $\lcm(m,n)$ denote the
greatest common divisor and the least common multiple of $m,n\in \N$, respectively;
$\Z_n$ denotes the set of residue classes modulo $n$ ($n\in \N$); $\tau(n)$ is the number of divisors of $n$;
$\varphi$ is Euler's arithmetic function.

Consider the ring $(\Z_m\times \Z_n,+,\cdot)$, where $m,n\in \N$. If $\gcd(m,n)=1$, then it is isomorphic to the ring
$(\Z_{mn},+,\cdot)$. Hence, all of its additive subgroups are subrings, i.e., are closed under multiplication. In fact, all
additive subgroups are ideals of the given ring. If $\gcd(m,n)>1$, then this is not the case.
For example, $K:=\{(2i, i+3j): 0\le i,j\le 5 \}$ is an additive subgroup of
$\Z_{12}\times \Z_{18}$, $(2,7)\in K$, $(4,5) \in K$, but $(2,7)(4,5)=(8,17)\notin K$. At the same time, the subgroup
$L:=\{(2i,2i+3j): 0\le i,j\le 5 \}$ is a subring of $\Z_{12}\times \Z_{18}$, as a direct check shows. Here $L$ is not an ideal,
since, e.g., $(2,5)\in L$, but $(2,5)(1,3)=(2,15)\notin L$.

Therefore, the following natural questions arise: Let $m,n\in \N$. What are the subrings of the
ring $\Z_m \times \Z_n$? What are its unital subrings, i.e., subrings including the multiplicative
unity $(1,1)$? What are its ideals? What are the number of subrings, unital
subrings, respectively ideals of $\Z_m \times \Z_n$?

Subrings and ideals of direct products of rings were investigated by Anderson and Camillo \cite{AndCam2009},
Anderson and Kintzinger \cite{AK2008}, Chajda, Eigenthaler and L\"{a}nger \cite{CEL2018}. Versions of Goursat's
lemma for ideals and subrings of a direct product of rings were given in \cite[Th.\ 11]{AndCam2009}.
The ideals of the ring $\Z_m \times \Z_n$ were also discussed in a recent paper by
Chebolu and Henry \cite{CheHen2016}. In fact, the ideals of $\Z_m \times \Z_n$ are exactly of the form $I\times J$, where $I$ and $J$ are
additive subgroups of $\Z_m$ and $\Z_n$, respectively. This follows from a well known property concerning the ideals of the direct product
of two arbitrary rings with unity, and has a simple proof. See \cite[Prop.\ 9]{AndCam2009}.
Hence, the number of ideals of $\Z_m \times \Z_n$ is $\tau(m)\tau(n)$.

However, there are no direct results in the above papers concerning the subrings and unital subrings of
$\Z_m \times \Z_n$, and we are not aware of related results in the literature. We remark that the ideals and subrings of the
ring $m\Z\times  n\Z$ were discussed in \cite[Ex.\ 12]{AndCam2009}. Subrings and unital subrings
(called sublattices and subrings, respectively) of a fixed index of the ring $\Z^n$ were studied by Liu \cite{Liu2007},
and in recent preprints by Atanasov, Kaplan, Krakoff and Menzel \cite{AKKM2017}, Chimni and Takloo-Bighash \cite{ChiTak2018}.

The following results on the representation and the number of subgroups of the group
$(\Z_m\times \Z_n,+)$ with $m,n\in \N$ were deduced by the author \cite{Tot2014Tatra}, using
Goursat's lemma for groups. See also the papers \cite{HHTW2014,Tar2010}, using different approaches.

\begin{theorem}[{\rm \cite[Th.\ 3.1]{Tot2014Tatra}}] \label{Th_repres} Let $m,n\in \N$.
For every $m,n\in \N$ let
\begin{equation*} \label{def_J}
J_{m,n}:=\left\{(a,b,c,d,\ell)\in \N^5: a\mid m, b\mid a, c\mid n,
d\mid c, \frac{a}{b}=\frac{c}{d}, \ell \le \frac{a}{b}, \,
\gcd\left(\ell,\frac{a}{b} \right)=1\right\}.
\end{equation*}

For $(a,b,c,d,\ell)\in J_{m,n}$ define
\begin{equation} \label{def_K}
K_{a,b,c,d,\ell}:= \left\{\left(i\frac{m}{a}, i\ell
\frac{n}{c}+j\frac{n}{d}\right): 0\le i\le a-1, 0\le j\le
d-1\right\}.
\end{equation}

Then the following hold true:

i) The map $(a,b,c,d,\ell)\mapsto K_{a,b,c,d,\ell}$ is a bijection
between the set $J_{m,n}$ and the set of subgroups of $(\Z_m \times
\Z_n,+)$.

ii) The invariant factor decomposition of the subgroup $K_{a,b,c,d,\ell}$ is
\begin{equation*} \label{H_isom}
K_{a,b,c,d,\ell} \simeq \Z_{\gcd(b,d)} \times \Z_{\lcm(a,c)}.
\end{equation*}

iii) The order of the subgroup $K_{a,b,c,d,\ell}$ is $ad$ and its
exponent is $\lcm(a,c)$.

iv) The subgroup $K_{a,b,c,d,\ell}$ is cyclic if and only if
$\gcd(b,d)=1$.
\end{theorem}

We note that by the condition $a/b=c/d$ we have $\lcm(a,c)=\lcm(a,ad/b)=\lcm(ad/d,ad/b)=ad/\gcd(b,d)$.
That is, $\gcd(b,d)\cdot \lcm(a,c)=ad$. Also, $\gcd(b,d)\mid \lcm(a,c)$.

Figure 1 represents the subgroup $K_{6,2,18,6,2}$ of $\Z_{12}\times \Z_{18}$. It has order $36$ and is
isomorphic to $\Z_2\times \Z_{18}$. Here $K_{6,2,18,6,2}= \{(2i,2i+3j): 0\le i,j\le 5 \}=L$, quoted above, and it is
also a subring.

\begin{theorem}[{\rm \cite[Th.\ 4.1]{Tot2014Tatra}}] \label{Th_number} Let $m,n\in \N$. The number $s(m,n)$ of
subgroups of the group $(\Z_m \times \Z_n,+)$ is given by
\begin{equation*} \label{total_number_subgroups}
s(m,n)= \sum_{\substack{i\mid m\\ j\mid n}} \gcd(i,j)
= \sum_{t \mid \gcd(m,n)} \varphi(t) \tau \left(\frac{m}{t} \right)
\tau \left(\frac{n}{t} \right).
\end{equation*}
\end{theorem}

\medskip \medskip
\centerline{Figure 1. The subgroup $K_{6,2,18,6,2}$ of $\Z_{12}\times \Z_{18}$}
\medskip
\centerline{$
\begin{array}{ccccccccccccc}
17 & \cdot & \cdot & \blacksquare & \cdot & \cdot & \cdot & \cdot & \cdot & \blacksquare & \cdot & \cdot & \cdot \\
16 & \cdot & \cdot & \cdot & \cdot & \blacksquare & \cdot & \cdot & \cdot & \cdot & \cdot & \blacksquare & \cdot \\
15 & \blacksquare & \cdot & \cdot & \cdot & \cdot & \cdot & \blacksquare & \cdot & \cdot & \cdot & \cdot & \cdot \\
14 & \cdot & \cdot & \blacksquare & \cdot & \cdot & \cdot & \cdot & \cdot & \blacksquare & \cdot & \cdot & \cdot \\
13 & \cdot & \cdot & \cdot & \cdot & \blacksquare & \cdot & \cdot & \cdot & \cdot & \cdot & \blacksquare & \cdot \\
12 & \blacksquare & \cdot & \cdot & \cdot & \cdot & \cdot & \blacksquare & \cdot & \cdot & \cdot & \cdot & \cdot \\
11 & \cdot & \cdot & \blacksquare & \cdot & \cdot & \cdot & \cdot & \cdot & \blacksquare & \cdot & \cdot & \cdot \\
10 & \cdot & \cdot & \cdot & \cdot & \blacksquare & \cdot & \cdot & \cdot & \cdot & \cdot & \blacksquare & \cdot \\
9 & \blacksquare & \cdot & \cdot & \cdot & \cdot & \cdot & \blacksquare & \cdot & \cdot & \cdot & \cdot & \cdot \\
8 & \cdot & \cdot & \blacksquare & \cdot & \cdot & \cdot & \cdot & \cdot & \blacksquare & \cdot & \cdot & \cdot \\
7 & \cdot & \cdot & \cdot & \cdot & \blacksquare & \cdot & \cdot & \cdot & \cdot & \cdot & \blacksquare & \cdot \\
6 & \blacksquare & \cdot & \cdot & \cdot & \cdot & \cdot & \blacksquare & \cdot & \cdot & \cdot & \cdot & \cdot \\
5 & \cdot & \cdot & \blacksquare & \cdot & \cdot & \cdot & \cdot & \cdot & \blacksquare & \cdot & \cdot & \cdot \\
4 & \cdot & \cdot & \cdot & \cdot & \blacksquare & \cdot & \cdot & \cdot & \cdot & \cdot & \blacksquare & \cdot \\
3 & \blacksquare & \cdot & \cdot & \cdot & \cdot & \cdot & \blacksquare & \cdot & \cdot & \cdot & \cdot & \cdot \\
2 & \cdot & \cdot & \blacksquare & \cdot & \cdot & \cdot & \cdot & \cdot & \blacksquare & \cdot & \cdot & \cdot \\
1 & \cdot & \cdot & \cdot & \cdot & \blacksquare & \cdot & \cdot & \cdot & \cdot & \cdot & \blacksquare & \cdot \\
0 & \blacksquare & \cdot & \cdot & \cdot & \cdot & \cdot & \blacksquare & \cdot & \cdot & \cdot & \cdot & \cdot \\
& 0\text{ } & 1\text{ } & 2\text{ } & 3\text{ } & 4\text{ } & 5\text{ } & 6
\text{ } & 7\text{ } & 8\text{ } & 9\text{ } & 10 & 11
\end{array}
$}
\medskip \medskip

We recall that a nonzero arithmetic function of
two variables $(m,n)\mapsto f(m,n)$ is said to be  multiplicative if
$f(m_1n_1,m_2n_2)=f(m_1,m_2)f(n_1,n_2)$, provided that $\gcd(m_1m_2,n_1n_2)=1$. We refer to our survey paper \cite{Tot2014Survey} regarding
this concept. If $f$ is multiplicative, then it is determined
by the values $f(p^\alpha,p^\beta)$, where $p$ is prime and $\alpha, \beta \in \N_0$. More exactly, $f(1,1)=1$ and
for any $m,n \in \N$,
\begin{equation*}
f(m,n)= \prod_p f(p^{\nu_p(m)},p^{\nu_p(n)}).
\end{equation*}

The function $(m,n)\mapsto s(m,n)$ is multiplicative and for any prime powers $p^\alpha, p^\beta$ with $1\le \alpha \le \beta$,
its values are given by
\begin{equation} \label{s_prime_pow}
s(p^\alpha,p^\beta)= \frac{(\beta-\alpha+1)p^{\alpha+2}-(\beta-\alpha-1)p^{\alpha+1}-(\alpha+\beta+3)p+(\alpha+\beta+1)}{(p-1)^2}.
\end{equation}

In this paper we characterize the subgroups $K_{a,b,c,d,\ell}$ of $\Z_m \times \Z_n$, given by \eqref{def_K}, which are also
(unital) subrings, and deduce explicit formulas for $N^{(s)}(m,n)$ and $N^{(us)}(m,n)$, denoting the number of subrings and unital subrings,
respectively. We show that the functions $(m,n)\mapsto N^{(s)}(m,n)$ and $(m,n)\mapsto N^{(us)}(m,n)$ are also multiplicative,
and their Dirichlet series can be expressed in terms of the Riemann zeta function.
We establish an asymptotic formula for the sum $\sum_{m,n\le x} N^{(s)}(m,n)$, the error term of which is closely related to the
Dirichlet divisor problem. Our results are included in  Section \ref{Section_Results} and their proofs are given in Section \ref{Section_Proofs}.

\section{Main results} \label{Section_Results}

\begin{theorem} \label{Th_1} Let $m,n\in \N$ and consider the additive subgroup $K_{a,b,c,d,\ell}$ of $\Z_m \times \Z_n$,
given by \eqref{def_K}.

i) The subgroup $K_{a,b,c,d,\ell}$ is a subring of $\Z_m \times \Z_n$ if and only if
\begin{equation} \label{cond_subring}
\frac{c}{d} \mid \ell\frac{n}{c} -\frac{m}{a}.
\end{equation}

ii) The subgroup $K_{a,b,c,d,\ell}$ is a unital subring of $\Z_m \times \Z_n$ if and only if $a=m$, $c=n$, $\ell=1$.
Here $K_{m,b,n,d,1}= \{\left(i,i+j\frac{n}{d}\right): 0\le i\le m-1, 0\le j\le d-1\}$, where $b\mid m$, $d\mid n$,
$\frac{m}{b}=\frac{n}{d}$.

iii) The subgroup $K_{a,b,c,d,\ell}$ is an ideal of $\Z_m \times \Z_n$ if and only if $a=b$, $c=d$, $\ell=1$. In this case
$K_{a,a,c,c,1}=\{i\frac{m}{a}: 0\le i\le a-1\} \times \{j\frac{n}{c}: 0\le j\le c-1\}$, where $a\mid m$, $c\mid n$.
\end{theorem}

\begin{theorem} \label{Th_2} Let $m,n\in \N$.

i) The number of subrings of the ring $\Z_m \times \Z_n$ is given by
\begin{equation} \label{def_N}
N^{(s)}(m,n)= \sum_{\substack{i\mid m\\ j\mid n}} h(i,j),
\end{equation}
where
\begin{equation} \label{def_h}
h(i,j)=\sum_{\substack{d\mid \gcd(i,j)\\ \gcd(d,i/d)=\gcd(d,j/d)=t}} \frac{\varphi(d)}{\varphi(d/t)}.
\end{equation}

ii) The number of unital subrings of the ring $\Z_m \times \Z_n$ is $N^{(us)}(m,n) = \tau(\gcd(m,n))$.

iii) The number of ideals of the ring $\Z_m \times \Z_n$ is  $\tau(m)\tau(n)$.
\end{theorem}

It is clear that $N^{(s)}(m,n)$ and $N^{(us)}(m,n)$ are symmetric in the variables.

\begin{theorem} \label{Th_3} The functions $(m,n)\mapsto h(m,n)$, $(m,n)\mapsto N^{(s)}(m,n)$ and $(m,n)\mapsto N^{(us)}(m,n)$
are multiplicative, viewed as arithmetic functions of two variables. For any prime powers $p^{\alpha},p^{\beta}$
\textup{($0\le \alpha \le \beta$)}
their values are given by the following polynomials in $p$:
\begin{equation} \label{form_comp_h}
h(p^{\alpha},p^{\beta})=
\begin{cases} \frac{p^{\gamma+1}-1}{p-1}, & \alpha=\beta=2\gamma, \\
p^{\gamma}, & \alpha= 2\gamma< \beta, \\
\frac{2p^{\gamma}-p^{\gamma-1}-1}{p-1}, & \alpha= \beta=2\gamma-1, \\
p^{\gamma-1}, & \alpha= 2\gamma-1< \beta,
\end{cases}
\end{equation}
\begin{equation} \label{form_comp_N}
N^{(s)}(p^{\alpha},p^{\beta})=
\begin{cases}
p^{\gamma} + 10 \sum_{j=1}^{\gamma} j p^{\gamma-j} \\
=\frac1{(p-1)^2} \left(p^{\gamma+2}+8p^{\gamma+1}+p^\gamma-10(\gamma+1)p+ 10\gamma \right), & \alpha=\beta=2\gamma, \\
(\beta-2\gamma+1)p^{\gamma}+ 2\sum_{j=1}^{\gamma} (5j+\beta-2\gamma)p^{\gamma-j}\\
= \frac1{(p-1)^2} \left((\beta-2\gamma+1)p^{\gamma+2}+ 8p^{\gamma+1}-(\beta-2\gamma-1)p^{\gamma} \right. \\
\left. -2(\beta+3\gamma+5)p+2(\beta-3\gamma)\right),
& \alpha= 2\gamma< \beta, \\
5 \sum_{j=1}^{\gamma} (2j-1) p^{\gamma-j}\\
= \frac{5}{(p-1)^2} \left(p^{\gamma+1}+p^{\gamma} -(2\gamma+1)p+ 2\gamma -1 \right), & \alpha= \beta=2\gamma-1, \\
\sum_{j=1}^{\gamma} (10j+2\beta-4\gamma-3) p^{\gamma-j}\\
= \frac1{(p-1)^2} \left((2\beta-4\gamma+7)p^{\gamma+1}-(2\beta-4\gamma-3)p^{\gamma}  \right. \\
\left. -(2\beta+6\gamma +7)p +2\beta+6\gamma-3\right), & \alpha= 2\gamma-1< \beta,
\end{cases}
\end{equation}
\begin{equation*}
N^{(us)}(p^{\alpha},p^{\beta})=\min(\alpha,\beta)+1.
\end{equation*}
\end{theorem}

Remark that for every $m,n\in \N$,
\begin{equation*}
 N^{(us)}(m,n)= \tau(\gcd(m,n))\le \tau(mn)\le \tau(m)\tau(n) \le N^{(s)}(m,n)\le s(m,n),
\end{equation*}
which follow from the definitions and properties of the $\tau$ function. Here $\tau(mn)$ is the number of cyclic
subproducts of $\Z_m \times \Z_n$. See \cite[Th.\ 5]{HHTW2014}.

To illustrate the applicability of our results, we note that the ring $\Z_{12}\times \Z_{18}$ has
\begin{equation*}
N^{(s)}(12,18)= N^{(s)}(2^2\cdot 3,2\cdot 3^2)= N^{(s)}(2,2^2) N^{(s)}(3,3^2)=7\cdot 7=49
\end{equation*}
subrings and $N^{(us)}(12,18)=\tau(\gcd(12,18))=\tau(6)=4$ unital subrings. The number of its
ideals is $\tau(12)\tau(18)=36$. The number of subgroups is $s(12,18)= s(2,2^2) s(3,3^2)=8\cdot 10 =80$,
by using \eqref{s_prime_pow}.

Now consider the case $m=n$. Let $h(n)=h(n,n)$. Furthermore, let $N^{(s)}(n):=N^{(s)}(n,n)$ and $N^{(us)}(n):=N^{(us)}(n,n)$ denote the
number of subrings, respectively unital subrings of the ring $\Z_n^2$.

\begin{cor} The functions $n\mapsto h(n)$, $n\mapsto N^{(s)}(n)$ and $n\mapsto N^{(us)}(n)$ are multiplicative, viewed as arithmetic functions
of one variable. For every $n\in \N$ we have
\begin{equation*}
N^{(s)}(n)= \sum_{i\mid n} h(i),
\end{equation*}
where
\begin{equation*}
h(i)=\sum_{d\mid i} \frac{\varphi(d)}{\varphi(d/\gcd(d,i/d))}.
\end{equation*}
and $N^{(us)}(n) = \tau(n)$.
\end{cor}

Let $\theta$ be the exponent in the Dirichlet divisor problem for $\tau(n)$, that is
\begin{equation} \label{Dirichlet}
\sum_{n\le x} \tau(n) = x\log x + (2C-1)x +O(x^{\theta+\varepsilon})
\end{equation}
for every $\varepsilon >0$, where $C$ is Euler's constant. Note that $1/4\le \theta \le 517/1648\doteq 0.313713$,
the upper bound being a recent result due to Bourgain and Watt \cite{BouWat2017}.

\begin{theorem} \label{Th_asympt}  i)  For $z,w\in \C$ with $\Re z>1$, $\Re w>1$ we have
\begin{equation} \label{Dir_series_N}
\sum_{m,n=1}^{\infty} \frac{N^{(s)}(m,n)}{m^z n^w}=\frac{\zeta^2(z)\zeta^2(w)\zeta(z+w)\zeta(2z+2w-1)}{\zeta(z+2w)\zeta(2z+w)}.
\end{equation}

ii) For every $\varepsilon >0$,
\begin{equation} \label{asymp_subrings}
\sum_{m,n\le x} N^{(s)}(m,n)= x^2\left(A_1 \log^2 x +A_2\log x+A_3\right)+ O(x^{1+\theta+\varepsilon}),
\end{equation}
where $A_1=\zeta(2)/\zeta(3)$, $A_2,A_3$ are explicit constants.
\end{theorem}

The formulas \eqref{Dir_series_N} and \eqref{asymp_subrings} may be compared to
\begin{equation*}
\sum_{m,n=1}^{\infty} \frac{s(m,n)}{m^z n^w}=\frac{\zeta^2(z)\zeta^2(w)\zeta(z+w-1)}{\zeta(z+w)}
\end{equation*}
and
\begin{equation} \label{sum_s_m_n}
\sum_{m,n\le x} s(m,n) = x^2\left(B_1\log^3 x +B_2\log^2 x+ B_3\log x+B_4\right) + O(x^{\frac{3-\theta}{2-\theta}+\varepsilon}),
\end{equation}
respectively, where $B_1=2/\pi^2$, $B_2,B_3,B_4$ are explicit constants, proved by Nowak and the author \cite[Th.\ 2.1, 2.2]{NT}.
The error term of \eqref{sum_s_m_n} was improved into $O(x^{3/2}(\log x)^{13/2})$ by the author and Zhai \cite[Th.\ 1.1]{TZ}.
We also remark that for $N^{(us)}(m,n)=\tau(\gcd(m,n))$,
\begin{equation*}
\sum_{m,n=1}^{\infty} \frac{N^{(us)}(m,n)}{m^z n^w}=\zeta(z)\zeta(w)\zeta(z+w),
\end{equation*}
and
\begin{equation*}
\sum_{m,n\le x} N^{(us)}(m,n)= \zeta(2)x^2 + O(x\log x),
\end{equation*}
which easily follows from the identity $\sum_{m,n\le x} \tau(\gcd(m,n))=\sum_{d\le x} [x/d]^2$.

\section{Proofs} \label{Section_Proofs}

\begin{proof}[Proof of Theorem {\rm \ref{Th_1}}]
 i) Let $0\le i_1,i_2\le a-1, 0\le j_1,j_2\le d-1$. Assume that
\begin{equation*}
\left(i_1\frac{m}{a}, i_1\ell \frac{n}{c}+j_1\frac{n}{d}\right) \left(i_2\frac{m}{a}, i_2\ell
\frac{n}{c}+j_2\frac{n}{d}\right)
\end{equation*}
\begin{equation*}
= \left(i_1i_2\left(\frac{m}{a}\right)^2, i_1i_2\ell^2 \left(\frac{n}{c}\right)^2+(i_1j_2+i_2j_1) \ell
\frac{n^2}{cd} + j_1j_2\left(\frac{n}{d}\right)^2\right)
\end{equation*}
equals
\begin{equation*}
\left(i\frac{m}{a}, i\ell \frac{n}{c}+j\frac{n}{d}\right)
\end{equation*}
for some $i$ and $j$. This holds true if and only if
\begin{equation*}
i\frac{m}{a} \equiv i_1 i_2\left(\frac{m}{a}\right)^2 \quad \textup{(mod $m$)},
\end{equation*}
that is
\begin{equation} \label{cond_1}
i \equiv i_1 i_2 \frac{m}{a} \quad \textup{(mod $a$)},
\end{equation}
and
\begin{equation} \label{cond_2}
i\ell \frac{n}{c}+j\frac{n}{d} \equiv i_1i_2\ell^2 \left(\frac{n}{c}\right)^2+(i_1j_2+i_2j_1) \ell \frac{n^2}{cd} + j_1j_2\left(\frac{n}{d}\right)^2 \quad \textup{(mod $n$)}.
\end{equation}

By using \eqref{cond_1} and the fact that $\frac{n}{d}=\frac{n}{c}\cdot \frac{c}{d}$, \eqref{cond_2}
is equivalent to
\begin{equation} \label{cond_3}
\left(i_1i_2 \frac{m}{a}+ka\right)\ell  + j\frac{c}{d} \equiv i_1i_2\ell^2 \frac{n}{c}+(i_1j_2+i_2j_1) \ell
\frac{n}{d} + j_1j_2\frac{nc}{d^2} \quad \textup{(mod $c$)},
\end{equation}
with some $k\in \Z$.

Here \eqref{cond_3} is a linear congruence in $j$, and since $\gcd(c/d,c)=c/d$, it has solutions in $j$ if
and only if
\begin{equation*}
\frac{c}{d} \mid i_1i_2 \ell^2 \frac{n}{c}-(i_1i_2\frac{m}{a} +ka) \ell,
\end{equation*}
which is equivalent, by using that $\gcd(\ell, c/d)=1$ and $c/d=a/b\mid a$, to
\begin{equation*}
\frac{c}{d} \mid i_1i_2 \left(\ell \frac{n}{c}- \frac{m}{a}\right).
\end{equation*}

For $i_1=i_2=1$ we have the necessary condition $\frac{c}{d} \mid \ell \frac{n}{c}- \frac{m}{a}$,
which is also sufficient for every $i_1,i_2$.

ii) The subgroup $K_{a,b,c,d,\ell}$, given by \eqref{def_K}, contains the identity $(1,1)$ if and only if
\begin{equation*}
\left(i\frac{m}{a}, i\ell \frac{n}{c}+j\frac{n}{d}\right)=(1,1)
\end{equation*}
for some $i$ and $j$. This holds true if and only if $a=m$, $c=n$ and $\ell=1$ (with $i=1$, $j=0$). In this case, condition
\eqref{cond_subring} is satisfied.

iii) Let $R$ and $S$ be two commutative rings with identity. Then every ideal of $R\times S$ is of the form $I\times J$, where
$I$ and $J$ are ideals of $R$ and $S$, respectively. This follows from \cite[Prop.\ 9]{AndCam2009}, as already refered in
Section \ref{Section_Motivation}. Since the ideals of $\Z_m$ are its subgroups, we deduce that $K_{a,b,c,d,\ell}$ is an ideal if and only
if it is a subproduct of $\Z_m\times \Z_n$. That is, $K_{a,b,c,d,\ell}= I\times J$, where
$I$ and $J$ are subgroups of $\Z_m$ and $\Z_n$, respectively. Then $J$ is the second projection of $K_{a,b,c,d,\ell}$, namely
$\{jn/d: 0\le j\le d-1\}\le \Z_n$. Also, for every $i$ and $j$ there is $k$ such that $i\ell \frac{n}{c}+j\frac{n}{d}=k\frac{n}{d}$. For $i=1$, $j=0$
we get $\ell \frac{n}{c}=k\frac{n}{d}$, that is $\ell=k\frac{c}{d}$, where $d\mid c$. Using that $\gcd(\ell,\frac{c}{d})=1$ and $\ell \le \frac{c}{d}$,
we deduce that $c=d$, hence $a=b$, and $\ell=1$.
\end{proof}

To prove Theorem \ref{Th_2} we need the following known auxiliary result.
See \cite[Th.\ 2]{GP2013} and \cite[Th.\ 3.1]{BKST2017} for its proofs.

\begin{lemma} \label{Lemma} Let $n\in \N$, $a,b\in \Z$. The linear congruence $ax\equiv b$ (mod $n$) admits solutions
$x$ such that $\gcd(x,n)=1$ if and only if $\gcd(a,n)=\gcd(b,n)=d$. In this case the number of
solutions (mod $n$) is $\varphi(n)/\varphi(n/d)$.
\end{lemma}

\begin{proof}[Proof of Theorem {\rm \ref{Th_2}}]

i) By i) of Theorem \ref{Th_1} the number of subrings is
\begin{equation} \label{N}
N^{(s)}(m,n)= \sum_{\substack{a\mid m\\ b\mid a}} \sum_{\substack{c\mid
n\\ d\mid c}} \sum_{\substack{\ell=1\\ \gcd(\ell,e)=1}}^e  \sum_{\substack{a/b=c/d=e\\ e \mid \ell \frac{n}{c}- \frac{m}{a}}} 1.
\end{equation}

Let $m=ax$, $a=by$, $n=cz$, $c=dt$. Then, by the condition
$a/b=c/d=e$ we have $y=t=e$. Rearranging the terms of \eqref{N},
\begin{equation*}
N^{(s)}(m,n)= \sum_{\substack{bxe=m\\ dze=n}} \sum_{\substack{\ell=1\\ \gcd(\ell,e)=1\\ z \ell \equiv x \ \textup{(mod $e$)}}}^e 1
\end{equation*}

Using Lemma \ref{Lemma}, we deduce that
\begin{equation*}
N^{(s)}(m,n)=  \sum_{\substack{bxe=m\\ dze=n\\ \gcd(x,e)=\gcd(z,e)=t}} \frac{\varphi(e)}{\varphi(e/t)}
\end{equation*}
\begin{equation*}
=  \sum_{\substack{bi=m\\ dj=n}} \sum_{\substack{xe=i\\ ze=j\\ \gcd(x,e)=\gcd(z,e)=t}} \frac{\varphi(e)}{\varphi(e/t)}
=  \sum_{\substack{i\mid m\\ j\mid n}}  h(i,j).
\end{equation*}

ii) By ii) of Theorem \ref{Th_1} the number of unital subrings is
\begin{equation*}
N^{(us)}(m,n)= \sum_{\substack{b\mid m\\ d\mid n\\ m/b=n/d}} 1 = \sum_{\substack{bj=m\\ dj=n}} 1 =\sum_{j\mid \gcd(m,n)} 1
=\tau(\gcd(m,n)).
\end{equation*}

iii) Follows at once by iii) of Theorem \ref{Th_1}.
\end{proof}

\begin{proof}[Proof of Theorem {\rm \ref{Th_3}}] First we show that the function $(m,n)\mapsto h(m,n)$ is multiplicative.
Let $\gcd(m_1m_2,n_1n_2)=1$. Then, by using that the function $(m,n)\mapsto \gcd(m,n)$ is multiplicative, we deduce that
\begin{equation*}
h(m_1n_1,m_2n_2)=\sum_{\substack{d\mid \gcd(m_1n_1,m_2n_2)\\ \gcd(d,m_1n_1/d)=\gcd(d,m_2n_2/d)=t}} \frac{\varphi(d)}{\varphi(d/t)}
\end{equation*}
\begin{equation*}
=\sum_{\substack{d\mid \gcd(m_1,m_2)\gcd(n_1,n_2) \\ \gcd(d,m_1n_1/d)=\gcd(d,m_2n_2/d)=t}} \frac{\varphi(d)}{\varphi(d/t)},
\end{equation*}
where $\gcd(m_1,m_2)$ and $\gcd(n_1,n_2)$ are relatively prime. Let $d=ab$ such that $a\mid \gcd(m_1,m_2)$ and $b\mid \gcd(n_1,n_2)$.
Then
\begin{equation*}
\gcd(d,m_1n_1/d)=\gcd(ab,m_1n_1/(ab)) =\gcd(a,m_1/a)\gcd(b,n_1/b)
\end{equation*}
and
\begin{equation*}
\gcd(d,m_2n_2/d)=\gcd(ab,m_2n_2/(ab))
=\gcd(a,m_2/a)\gcd(b,n_2/b).
\end{equation*}

Also, the condition $\gcd(d,m_1n_1/d)=\gcd(d,m_2n_2/d)$ implies that $\gcd(a,m_1/a)=\gcd(a,m_2/a)$ and
$\gcd(b,n_1/b)=\gcd(b,n_2/b)$. We deduce that
\begin{equation*}
h(m_1n_1,m_2n_2)= \sum_{\substack{a\mid \gcd(m_1,m_2) \\ \gcd(a,m_1/a)=\gcd(a,m_2/a)=A}} \frac{\varphi(a)}{\varphi(a/A)}
\sum_{\substack{b\mid \gcd(n_1,n_2) \\ \gcd(b,n_1/b)=\gcd(b,n_2/b)=B}} \frac{\varphi(b)}{\varphi(b/B)}
\end{equation*}
\begin{equation*}
= h(m_1,m_2)h(n_1,n_2).
\end{equation*}

Now it follows that the function $N^{(s)}(m,n)$  is also multiplicative, being the convolution of $h(m,n)$ with the constant $1$ function,
according to \eqref{def_N}.

Let $p^{\alpha}, p^{\beta}$ be any prime powers with $0\le \alpha \le \beta$. By \eqref{def_h} we have
 \begin{equation*}
h(p^{\alpha}, p^{\beta})= \sum_{\substack{d\mid \gcd(p^{\alpha},p^{\beta}) \\ \gcd(d,p^{\alpha}/d) =\gcd(d,p^{\beta}/d)=t}}
\frac{\varphi(d)}{\varphi(d/t)}
\end{equation*}
\begin{equation} \label{comp_h}
= \sum_{\substack{j=0 \\ \gcd(p^j,p^{\alpha-j}) =\gcd(p^j,p^{\beta-j})=t}}^{\alpha}
\frac{\varphi(p^j)}{\varphi(p^j/t)}.
\end{equation}

If $\alpha=\beta$, then
 \begin{equation*}
h(p^{\alpha}, p^{\alpha})= \sum_{j=0}^{\alpha}
\frac{\varphi(p^j)}{\varphi(p^{j-\min(j,\alpha-j)})},
\end{equation*}
which gives for $\alpha=2\gamma\ge 0$,
 \begin{equation*}
h(p^{2\gamma}, p^{2\gamma})= \sum_{j=0}^{\gamma} \varphi(p^j) + \sum_{j=1}^{\gamma} \frac{\varphi(p^{\gamma+j})}{\varphi(p^{2j})}
=p^{\gamma}+\sum_{j=1}^{\gamma} p^{\gamma-j}=\sum_{j=0}^{\gamma} p^{\gamma-j},
\end{equation*}
and for  $\alpha=2\gamma-1\ge 1$,
 \begin{equation*}
h(p^{2\gamma-1}, p^{2\gamma-1})= \sum_{j=0}^{\gamma-1} \varphi(p^j) + \sum_{j=1}^{\gamma}
\frac{\varphi(p^{\gamma+j-1})}{\varphi(p^{2j-1})}
=p^{\gamma-1}+\sum_{j=1}^{\gamma} p^{\gamma-j}.
\end{equation*}

By \eqref{comp_h}, if $\alpha=2\gamma<\beta$, then
\begin{equation*}
h(p^{2\gamma}, p^{\beta})= \sum_{j=0}^{\gamma} \varphi(p^j) = p^{\gamma},
\end{equation*}
and if $\alpha=2\gamma-1<\beta$, then
\begin{equation*}
h(p^{2\gamma-1}, p^{\beta})= \sum_{j=0}^{\gamma-1} \varphi(p^j) = p^{\gamma-1}.
\end{equation*}

This proves identities \eqref{form_comp_h}. From \eqref{def_N} we deduce that
\begin{equation*}
N^{(s)}(p^{\alpha},p^{\beta})= \sum_{j=0}^{\alpha} \sum_{k=0}^{\beta} h(p^j,p^k),
\end{equation*}
and by using \eqref{form_comp_h}, identities \eqref{form_comp_N}  are obtained by some direct computations, which are omitted.
\end{proof}

\begin{proof}[Proof of Theorem {\rm \ref{Th_asympt}}]
i) By Theorem \ref{Th_3} the function $h(m,n)$ is multiplicative. Therefore its Dirichlet series can be expanded into the Euler product
\begin{equation*}
\sum_{m,n=1}^{\infty} \frac{h(m,n)}{m^zn^w}= \prod_p \sum_{\alpha,\beta=0}^{\infty} \frac{h(p^{\alpha},p^{\beta})}{p^{\alpha z+
\beta w}}.
\end{equation*}

Taking into account  identities \eqref{form_comp_h} we deduce
\begin{equation*}
\sum_{m,n=1}^{\infty} \frac{h(m,n)}{m^zn^w}= \prod_p \left(\sum_{0\le \alpha=\beta} \frac{h(p^{\alpha},p^{\beta})}{p^{\alpha z+
\beta w}} + \sum_{0\le \alpha<\beta} \frac{h(p^{\alpha},p^{\beta})}{p^{\alpha z+ \beta w}}  + \sum_{0\le \beta<\alpha} \frac{h(p^{\alpha},p^{\beta})}{p^{\alpha z+ \beta w}} \right),
\end{equation*}
where
\begin{equation*}
\sum_{0\le \alpha=\beta} \frac{h(p^{\alpha},p^{\beta})}{p^{\alpha z+ \beta w}} =
\sum_{\gamma=0}^{\infty} \frac{h(p^{2\gamma},p^{2\gamma})}{p^{2\gamma(z+w)}}
+ \sum_{\gamma=1}^{\infty} \frac{h(p^{2\gamma-1},p^{2\gamma-1})}{p^{(2\gamma-1)(z+w)}}
\end{equation*}
\begin{equation*}
= \sum_{\gamma=0}^{\infty} \frac{\frac{p^{\gamma+1}-1}{p-1}}{p^{2\gamma(z+w)}}
+ \sum_{\gamma=1}^{\infty} \frac{\frac{2p^{\gamma}-p^{\gamma-1}-1}{p-1}}{p^{(2\gamma-1)(z+w)}}
\end{equation*}
\begin{equation} \label{h1}
=\left(1-\frac1{p^{z+w}}\right)^{-1} \left(1-\frac1{p^{2z+2w-1}}\right)^{-1} \left(1+\frac1{p^{z+w}}-\frac1{p^{2z+2w}}\right),
\end{equation}
\begin{equation*}
\sum_{0\le \alpha<\beta} \frac{h(p^{\alpha},p^{\beta})}{p^{\alpha z+ \beta w}} =
\sum_{0\le 2\gamma <\beta} \frac{h(p^{2\gamma},p^{\beta})}{p^{2\gamma z+ \beta w}} +
\sum_{0<2\gamma-1 <\beta} \frac{h(p^{2\gamma-1},p^{\beta})}{p^{(2\gamma-1) z+ \beta w}}
\end{equation*}
\begin{equation*}
=\sum_{\gamma=0}^{\infty} \sum_{\beta=2\gamma+1}^{\infty} \frac{p^{\gamma}}{p^{2\gamma z+\beta w}}
+ \sum_{\gamma=1}^{\infty} \sum_{\beta=2\gamma}^{\infty} \frac{p^{\gamma-1}}{p^{(2\gamma-1)z+\beta w}}
\end{equation*}
\begin{equation} \label{z_w_first}
=\left(1-\frac1{p^{w}}\right)^{-1} \left(1-\frac1{p^{2z+2w-1}}\right)^{-1} \left(\frac1{p^{w}}+\frac1{p^{z+2w}}\right),
\end{equation}
and by changing the role of $z$ and $w$ we get from \eqref{z_w_first} the identity
\begin{equation} \label{z_w_second}
\sum_{0\le \beta <\alpha} \frac{h(p^{\alpha},p^{\beta})}{p^{\alpha z+ \beta w}}
=\left(1-\frac1{p^{z}}\right)^{-1} \left(1-\frac1{p^{2z+2w-1}}\right)^{-1} \left(\frac1{p^{z}}+\frac1{p^{2z+w}}\right).
\end{equation}

Putting together \eqref{h1}, \eqref{z_w_first} and \eqref{z_w_second} we conclude that
\begin{equation*}
\sum_{m,n=1}^{\infty} \frac{h(m,n)}{m^zn^w}=
\prod_p \left(1-\frac1{p^{z}}\right)^{-1} \left(1-\frac1{p^{w}}\right)^{-1} \left(1-\frac1{p^{z+w}}\right)^{-1}
\left(1-\frac1{p^{2z+2w-1}}\right)^{-1}
\end{equation*}
\begin{equation*}
\times \left(1-\frac1{p^{z+2w}}\right) \left(1-\frac1{p^{2z+w}}\right)
\end{equation*}
\begin{equation} \label{Dir_serr_h}
=\frac{\zeta(z)\zeta(w)\zeta(z+w)\zeta(2z+2w-1)}{\zeta(z+2w)\zeta(2z+w)}.
\end{equation}

Now the Dirichlet series representation \eqref{Dir_series_N} of the function $N^{(s)}(m,n)$ is obtained from \eqref{Dir_serr_h}
by using that $N^{(s)}(m,n)$ is the convolution of $h(m,n)$ with the constant $1$ function.

ii) From  \eqref{Dir_series_N} we deduce the convolutional identity
\begin{equation} \label{convo_id}
N^{(s)}(m,n)= \sum_{\substack{ab=m\\ cd=n}} f(a,c)\tau(b)\tau(d),
\end{equation}
valid for every $m,n\in \N$, where
\begin{equation} \label{series_f}
 \sum_{m,n=1}^{\infty} \frac{f(m,n)}{m^zn^w}= \frac{\zeta(z+w)\zeta(2z+2w-1)}{\zeta(z+2w)\zeta(2z+w)}.
\end{equation}

We will use that the series \eqref{series_f} is absolutely convergent provided that $\Re z>0$, $\Re w>1$, $\Re(z+w)>1$.
We obtain by \eqref{convo_id},
\begin{equation*}
\sum_{m,n\le x} N^{(s)}(m,n)= \sum_{\substack{ab\le x\\ cd\le x}} f(a,c)\tau(b)\tau(d)= \sum_{a,c\le x} f(a,c)
\left(\sum_{b\le x/a} \tau(b)\right) \left(\sum_{d\le x/c} \tau(d)\right).
\end{equation*}

Using formula \eqref{Dirichlet} and by denoting $C_1=2C-1$, $\theta_1=\theta+\varepsilon$ we deduce
\begin{equation*}
\sum_{m,n\le x} N^{(s)}(m,n)
\end{equation*}
\begin{equation*}
=  \sum_{a, c\le x} f(a,c)
\left(\frac{x}{a}\log \frac{x}{a} + C_1\frac{x}{a}+O\left(\left(\frac{x}{a}\right)^{\theta_1}\right) \right)
\left(\frac{x}{c}\log \frac{x}{c} + C_1\frac{x}{c}+O\left(\left(\frac{x}{c}\right)^{\theta_1}\right) \right)
\end{equation*}
\begin{equation*}
=  x^2(\log^2 x+2C_1\log x+C_1^2) \sum_{a, c\le x} \frac{f(a,c)}{ac}  - x^2(\log x+C_1) \sum_{a, c\le x} \frac{f(a,c)(\log a+\log c)}{ac}
\end{equation*}
\begin{equation*}
+ x^2\sum_{a, c\le x} \frac{f(a,c)(\log a)(\log c)}{ac}
\end{equation*}
\begin{equation} \label{errors}
+O\left(x^{1+\theta_1} \sum_{a, c\le x} \frac{f(a,c)}{a^{\theta_1}c} \right) + O\left(x^{1+\theta_1} \sum_{a, c\le x}
\frac{f(a,c)}{ac^{\theta_1}}  \right).
\end{equation}

Here
\begin{equation*}
\sum_{a, c\le x} \frac{f(a,c)}{ac} =\sum_{a, c=1}^{\infty} \frac{f(a,c)}{ac} -\sideset{}{'} \sum_{a, c} \frac{f(a,c)}{ac},
\end{equation*}
where $\sum_{a, c}^{'}$ means that $a>x$ or $c>x$ (or both), and by \eqref{series_f},
\begin{equation*}
\sum_{a, c=1}^{\infty} \frac{f(a,c)}{ac}= \frac{\zeta(2)}{\zeta(3)}
\end{equation*}

To handle the sum $\sum_{a,c}^{'}$ we may assume that $c>x$ and have
\begin{equation*}
\sideset{}{'} \sum_{\substack{a,c\\ c> x}} \frac{f(a,c)}{ac} \ll \sideset{}{'} \sum_{\substack{a,c\\ c> x}} \frac{|f(a,c)|}{ac^{\varepsilon}}
\cdot \frac1{c^{1-\varepsilon}}\le \frac1{x^{1-\varepsilon}} \sum_{a,c=1}^{\infty} \frac{|f(a,c)|}{ac^{\varepsilon}}\ll   \frac1{x^{1-\varepsilon}},
\end{equation*}
since the latter series is convergent. In a similar way,
\begin{equation*}
\sum_{a, c\le x} \frac{f(a,c)(\log a+\log c)}{ac} = D+O\left(\frac1{x^{1-\varepsilon}}\right),
\end{equation*}
\begin{equation*}
\sum_{a, c\le x} \frac{f(a,c)(\log a)(\log c)}{ac} = E+O\left(\frac1{x^{1-\varepsilon}}\right),
\end{equation*}
with some explicit constants $D$ and $E$. Finally, both error terms in \eqref{errors} are $O(x^{1+\theta_1})$. This completes the proof.
\end{proof}

\end{document}